\documentclass[12pt]{amsart}
   \usepackage{amsmath,amsthm}
   \usepackage{amsfonts}   
   \usepackage{amssymb}    
   \usepackage{url}

\advance \topmargin by -\headheight
\advance \topmargin by -\headsep
      
\evensidemargin \oddsidemargin
\usepackage[all]{xy}

\newtheorem{thm}{Theorem}[section]
\newtheorem{prop}[thm]{Proposition}
\newtheorem{lem}[thm]{Lemma}
\newtheorem{cor}[thm]{Corollary}

\newtheorem{defn}[thm]{Definition}

\theoremstyle{remark}
\newtheorem{rem}[thm]{Remark}

\title{Flat structures on surface bundles}
\author{Jonathan Bowden}
\address{Mathematisches Institut, Ludwig-Maximilians-Universit\"at, Theresienstr. 39, 80333 M\"unchen, Germany}
\email{jonathan.bowden@mathematik.uni-muenchen.de}
\date{\today}
\begin{document}
\begin{abstract}
We show that there exist flat surface bundles with closed leaves having non-trivial normal bundles. This leads us to compute the Abelianisation of surface diffeomorphism groups with marked points. We also extend a formula of Tsuboi that expresses the Euler class of a flat circle bundle in terms of the Calabi invariant of certain Hamiltonian diffeomorphisms to surfaces of higher genus and derive a similar formula for the first MMM-class of surface bundles with punctured fibre.
\end{abstract}
\maketitle
\section{Introduction}
In this paper we study properties of the horizontal foliations of certain kinds of flat bundles. The paper divides into two main parts, the first of which is concerned with closed leaves of flat surface bundles and the second of which deals with fillings of flat circle bundles.

The existence of flat bundles with compact fibre that have closed leaves with non-trivial normal bundles is well known, with explicit examples given by certain flat structures of sphere bundles (cf.\ \cite{Mit}). This leads to the question of whether similar foliations exist on more general kinds of bundles. In view of this, we show the existence of foliated surface bundles with closed leaves that have prescribed self-intersection numbers, where the fibre can be taken to be any surface of sufficiently large genus. This is proved by a variation of the stabilisation trick of \cite{KM} that was originally used to show the existence of flat surface bundles with non-zero signatures. Indeed, given a surface bundle with a section $S$, whose self-intersection number is divisible by the Euler characteristic of the fibre, we show that $S$ can be made the leaf of a flat structure after stabilisation (Theorem \ref{leaf_stabilisation}).

To show the existence of flat surface bundles possessing closed leaves with non-trivial normal bundles, we are naturally led to certain calculations involving the homology of surface diffeomorphism groups with marked points. In particular, we show the following:
\begin{thm}
Let $\Sigma_h$ be a surface of genus $h \geq 3$ and let $k \geq 2$. Then 
\[H_1(Diff_{\delta}^+(\Sigma_{h,k})) = \mathbb{R}^+ \times \mathbb{Z}_2.\]
\end{thm}
\noindent We also show that the Abelianisation of the group of compactly supported diffeomorphism of $\mathbb{R}^2$ that fix the origin is $ \mathbb{R}^+$. This is a special case of a result of \cite{Fuk}, although the proof given there seems to overlook a small, but important, technical point (see discussion preceding Theorem \ref{Fukui}). Moreover, our proof, which relies on Sternberg Linearisation, is briefer than and independent of Fukui's original argument.

The stabilisation trick can also be applied to obtain foliations that have transverse symplectic structures. As in the case of flat bundles we show the existence of symplectically flat surface bundles with closed leaves having prescribed self-intersection numbers (Proposition \ref{symp_leaves_construction}). These results yield the existence of manifolds with certain kinds of symplectic pairs, as defined in \cite{KM}. Indeed, by applying the normal connected sum operation we deduce the existence of symplectic pairs both of whose foliations have closed leaves with non-trivial normal bundles (Corollary \ref{Pairs_no_triv_leaves}).

Given a flat circle bundle, one can consider the problem of extending the flat structure to the interior of a surface bundle. If the fibre is assumed to be a disc, then there is a dichotomy depending on whether one requires that the foliation is symplectic or not. For in the smooth case, it is relatively easy to show that any flat circle bundle over a surface admits a flat disc bundle filling after stabilisation (Proposition \ref{boundary_Euler_class}). However, in the symplectic case the Euler class provides an obstruction by a result of Tsuboi. In fact, Tsuboi gave the following formula for computing the Euler class of a flat circle bundle in terms of the Calabi homomorphism of certain extensions of the boundary holonomy to the interior of a disc:
\begin{thm}[\cite{Tsu}]\label{Tsuboi_int}
Let $\pi_1(\Sigma_g) \stackrel{\psi} \rightarrow Diff_0(S^1)$ be a homomorphism and let $a_i, b_i$ be standard generators of $\pi_1(\Sigma_g)$. Furthermore, let $f_i, h_i \in Symp(D^2)$ be extensions of $\psi(a_i), \psi(b_i)$ respectively and let $e(E)$ denote the Euler class of the total space of the $S^1$-bundle $E$. Then 
\[-\pi^2 e(E) = Cal([f_1, h_1]...[f_g, h_g]).\]
\end{thm}
\noindent Tsuboi's result can be reformulated in terms of the five-term exact sequence in group cohomology. The advantage of this reformulation is that it can easily be generalised to the case where the fibre of the filling is an arbitrary surface and the extensions of the boundary holonomies lie in the extended Hamiltonian group (Theorem \ref{Hamiltonian_extended_Tsuboi}).

As a consequence we see that the Euler class gives an obstruction to filling a circle bundle by a flat surface bundle with holonomy in the extended Hamiltonian group. We contrast this result with the fact that any flat circle bundle can be filled by a flat symplectic bundle after stabilisation (Theorem \ref{symp_flat_extension}). 

Another important topological quantity associated with a surface bundle is its signature. By considering appropriate extended Hamiltonian groups, we will derive a Tsuboi-type formula for the signature of bundles with fibre a once-punctured surface $\Sigma_h^1$.
\begin{thm}
Let $\Sigma^1_h \to E \to \Sigma_g$ be a bundle with holonomy representation $\pi_1(\Sigma_g) \stackrel{\rho} \rightarrow \Gamma_h^1$ and assume $h \geq 2$. Furthermore, let $\alpha_i = \rho(a_i)$ and $\beta_i = \rho(b_i)$ be the images of standard generators $a_i,b_i$ of $\pi_1(\Sigma_g)$. Then for any lifts $\phi_i, \psi_i \in \widetilde{Ham^c}(\Sigma_h^1)$ of $\alpha_i, \beta_i$ the signature satisfies
\[\sigma(E) = Cal([\phi_1, \psi_1]...[\phi_g, \psi_g ]).\]
\end{thm}
\noindent As is the case for fillings of circle bundles, this then implies restrictions on the topology of bundles whose holonomy groups lie in the extended Hamiltonian group.

\subsection*{Outline of paper:}
In Section \ref{closed_leaves_horizontal} we show the existence of flat bundles possessing closed leaves with non-trivial normal bundles and in Section \ref{points} we compute the Abelianisation of diffeomorphism groups with marked points. In Section \ref{symp_flat_leaves} we derive similar results for symplectic surface bundles. Then after showing that the problem of filling circle bundles is solvable after stabilisation in Section \ref{fill_poss}, we recast Tsuboi's result in the language of the five-term exact sequence and extend it to higher genus fillings in Section \ref{Flat_Ham}. In Section \ref{MMM_first} we derive the Tsuboi-type formula for the signature of surface bundles with once-punctured fibre. 

For the sake of completeness we give an explicit description of the five-term exact sequence in group cohomology in Appendix \ref{App_five_term}. 

\subsection*{Acknowledgments:}
The results of this article are taken from the author's doctoral thesis, which would not have been possible without the encouragement and support of Prof.\ D.\ Kotschick. The financial support of the Deutsche Forschungsgemeinschaft is also gratefully acknowledged.

\subsection*{Notation and Conventions:}
Throughout this paper $\Sigma^r_{g,k}$ will denote a compact, oriented genus $g$ surface with $k$ marked points and $r$ boundary components. By $\Gamma^r_{g,k}$ we shall denote the mapping class group of diffeomorphisms that preserve $k$ marked points and have support in the interior of $\Sigma^r_g$. 

All bundles will be assumed to be oriented and all maps are smooth. Unless otherwise stated all homology groups will be taken with integral coefficients. Finally a topological group will be decorated with a $\delta$ when it is to be considered as a discrete group.

\section{Closed leaves and horizontal foliations}\label{closed_leaves_horizontal}
Interesting examples of codimension two foliations come from considering the horizontal foliations of flat surface bundles. In this section we will focus on the closed leaves of such foliations. Any surface bundle $\Sigma_h \to E \to B$ determines a holonomy representation \[\pi_1(B) \stackrel{\rho} \rightarrow \Gamma_h.\]
Such a bundle is then \emph{flat} if its holonomy representation $\rho$ admits a lift to the group of orientation preserving diffeomorphisms $Diff^+(\Sigma_h)$
\[\xymatrix{ & Diff^+(\Sigma_h) \ar[d] \\
\pi_1(B) \ar[r]^{\rho} \ar@{-->}[ur]^{\bar{\rho}} & \Gamma_h.} \]
If $B$ is a manifold, then this is equivalent to the existence of a foliation that is complementary to the fibres. Such a foliation will be called a \emph{horizontal} foliation.

For a flat bundle $E$ over a manifold a closed leaf of the horizontal foliation intersects each fibre in $k$ points, where $k$ is the homological intersection number of the leaf with a fibre. The existence of such a foliation is thus equivalent to a lift of the holonomy map $\rho$ to the group of diffeomorphisms fixing $k$ marked points $Diff^+(\Sigma_{h,k})$
\[\xymatrix{ & Diff^+(\Sigma_{h,k}) \ar[d] \\
\pi_1(B) \ar[r]^{\rho} \ar@{-->}[ur]^{\bar{\rho}} & \Gamma_h.} \]
By taking pullbacks under a suitable finite cover of the base, one obtains a horizontal foliation with a leaf $S$ that intersects each fibre exactly once, in which case the holonomy of $E$ lies in $Diff^+(\Sigma_{h,1})$. Moreover, the horizontal foliation induces a flat structure on the normal bundle $\nu_S$ of $S$, which is given by composing $\bar{\rho}$ with the derivative map at $p$:
\[Diff^+(\Sigma_{h,1}) \stackrel{D_p} \rightarrow GL^+(T_p \Sigma_g) = GL^+(2, \mathbb{R}).\]

In \cite{Mil}, Milnor constructed flat bundles with non-trivial Euler class over oriented surfaces and, hence, the image of the Euler class in $H^2(GL_{\delta}^+(2, \mathbb{R}))$ is non-trivial. In view of this, to show that there are flat bundles with horizontal leaves of non-zero self-intersection it will suffice to show that the map $D_p$ induces an injection $H^2(GL_{\delta}^+(2,\mathbb{R})) \rightarrow H^2(Diff_{\delta}^+(\Sigma_{h,1}))$.

To this end, we let $\mathcal{G}_p$ be the group of smooth diffeomorphism germs that fix the marked point $p$. We then define
\[\xymatrix{Diff^+(\Sigma_{h,1}) \ar[r]^<<<<\pi & \mathcal{G}_p \ar[r]^<<<<<<{\bar{D}_p} & \ar @/^1pc/[l]^s GL^+(T_p \Sigma_g)},\]
where $\pi$ is the map taking a diffeomorphism fixing $p$ to its germ at $p$ and $\bar{D}_p$ maps a germ to its linear part. The kernel of the map $\pi$ consists of diffeomorphisms with support disjoint from the marked point $p$ and will be denoted by $Diff^c(\Sigma_{h,1})$. The final map has an obvious section given by considering a linear map as an element of $\mathcal{G}_p$. Thus, to show that the map $H^2(GL_{\delta}^+(2,\mathbb{R})) \rightarrow H^2(Diff_{\delta}^+(\Sigma_{h,1}))$ is injective, it will be sufficient to show that $H^2(\mathcal{G}_p) \rightarrow H^2(Diff_{\delta}^+(\Sigma_{h,1}))$ is injective. To this end we first note the following lemma, which is proved by a simple cut-off argument (cf.\ \cite{Bow2}, Lemma 4.1.3).
\begin{lem}\label{exact_sequence}
The following sequence of groups is exact
\[ 1 \to Diff^c(\Sigma_{h,1}) \to Diff^+(\Sigma_{h,1}) \rightarrow \mathcal{G}_p \to 1.\]
\end{lem}
\noindent We may now prove the existence of horizontal foliations that have compact leaves with non-trivial self-intersection numbers.
\begin{prop}\label{horizontal_leaves}
If $h \geq 3$, then there exist flat surface bundles $\Sigma_h \to E \to \Sigma_g$ with horizontal foliations that have leaves of non-zero self-intersection.
\end{prop}
\begin{proof}
We consider the last three terms of the five-term exact sequence in cohomology associated to the extension of groups in Lemma \ref{exact_sequence}:
\[ H^1(Diff_{\delta}^c(\Sigma_{h,1}))^{\mathcal{G}_p} \to H^2(\mathcal{G}_p) \stackrel{\pi^*} \rightarrow H^2(Diff^+_{\delta}(\Sigma_{h,1})).\]
From our discussion above it is sufficient to show that the map $\pi^*$ is injective or by exactness that $H^1(Diff_{\delta}^c(\Sigma_{h,1}))^{\mathcal{G}_p} = 0$. We claim that $H_1(Diff_{\delta}^c(\Sigma_{h,1}))$ is in fact trivial and by the Universal Coefficient Theorem the same holds in cohomology.

Let $\Sigma^{\epsilon}_h = \Sigma_h \setminus D_{\epsilon}$ denote $\Sigma_h$ with a disc of radius $\epsilon$ removed. We note that $Diff^c(\Sigma_{h,1})$ is isomorphic to the direct limit of the groups $Diff^c(\Sigma^{\epsilon}_h)$. By the stability result of Harer $H_1(\Gamma_h^1) = H_1(\Gamma_h)$ for $h \geq 3$ (cf.\ \cite{Iva}). Moreover, $\Gamma_h$ is perfect for $h \geq 3$ (see \cite{Pow}). Finally, by the classical result of Thurston the identity component of $Diff^c(\Sigma^{\epsilon}_{h})$ is also perfect (see \cite{Th2}). The five-term sequence in homology then implies that $H_1(Diff_{\delta}^c(\Sigma^{\epsilon}_h)) = 0$. Hence, each of the groups $H_1(Diff_{\delta}^c(\Sigma^{\epsilon}_h))$ is trivial and we conclude that $H_1(Diff_{\delta}^c(\Sigma_{h,1}))$ also vanishes.
\end{proof}
One may interpret the proof of Proposition \ref{horizontal_leaves} in a more geometric fashion, which gives a sharper result, and we note this in the following proposition. 

\begin{prop}\label{horizontal_leaves2}
If $h \geq 3$ and $k \in \mathbb{Z}$, then there exist flat surface bundles $\Sigma_h \to E \to \Sigma_g$ with horizontal foliations that have closed leaves of self-intersection $k$.
\end{prop}
\begin{proof}
Let $a_i, b_i \in \pi_1(\Sigma_g)$ denote the standard generators of the fundamental group and let $\xi_k$ be a flat $GL^+(2, \mathbb{R})$-bundle over $\Sigma_g$ that has Euler class $k \leq g-1$ as provided by \cite{Mil}. This corresponds to a holonomy representation
\[a_i \mapsto A_i\]
\[b_i \mapsto B_i\]
for $A_i, B_i \in GL^+(2, \mathbb{R})$. Then by Lemma \ref{exact_sequence} there are diffeomorphisms $\phi_i, \psi_i$ which agree with $A_i, B_i$ in a small neighbourhood of $p$ so that the product $\eta = \prod_{i = 1}^{g} [\phi_i, \psi_i]$ has support disjoint from $p$, that is $\eta \in Diff^c(\Sigma_{h,1})$. This group is perfect and, thus, we may write $\eta^{-1} = \prod_{i = 1}^{g'} [\alpha_i, \beta_i]$ where $\alpha_i, \beta_i \in  Diff^c(\Sigma_{h,1})$. We define a flat bundle over $\Sigma_{g + g'}$ by the holonomy representation
\[a_i \mapsto \phi_i, \ \ b_i \mapsto \psi_i \text{ for } 1 \leq i \leq g\]
\[a_{g +j} \mapsto \alpha_j, \ \ b_{g +j} \mapsto \beta_j \text{ for } 1 \leq j \leq g',\]
which we denote by $\rho$. This bundle then has a compact leaf $S$ corresponding to the marked point $p$. The Euler class of the normal bundle $\nu_S$ to $S$ is computed from the induced holonomy representation $D_p \rho$:
\[a_i \mapsto D_p(\phi_i) = A_i, \ \ b_i \mapsto D_p(\psi_i) = B_i \text{ for } 1 \leq i \leq g\]
\[a_{g +j} \mapsto D_p(\alpha_j) = Id, \ \ b_{g +j} \mapsto D_p(\beta_j) = Id \text{ for } 1 \leq j \leq g'.\]
It then follows that $e(\nu_S) = k$ and thus $[S]^2 = k$.
\end{proof}
With the geometric construction of Proposition \ref{horizontal_leaves2} we are now able to say when a section $S$ of a bundle $E$ can be made a leaf of a horizontal foliation after \emph{stabilisation}. 
\begin{defn}
 A surface bundle $E'$ over a surface is called a stabilisation of a bundle $E$, if it is the fibre sum of $E$ with a trivial bundle $\Sigma_{h} \times \Sigma_{g'}$. This is then a bundle over the connected sum $\Sigma_g \# \Sigma_{g'} = \Sigma_{g + g'}$ that is trivial over the second factor.
\end{defn} 
We will show that under certain conditions any bundle $E$ with a section $S$ of self-intersection $k$ can be stabilised to a bundle $E'$ that admits a horizontal foliation with a closed leaf $S'$ that agrees with $S$ on $E'|_{\Sigma_g}$. 

If the bundle $E$ is trivial then after stabilisation it remains trivial. For a trivial bundle the Euler class of the vertical tangent bundle $e(E)$ is divisible by $2h-2$ and, hence, the same is true for the self-intersection of $S$ and its stabilisation $S'$. Thus, the condition that $[S]^2$ is divisible by $2h-2$ is, in general, necessary for the existence of a stabilisation of the desired form. It is, however, also sufficient and this is the content of the following theorem.
\begin{thm}\label{leaf_stabilisation}
Let $\Sigma_h \to E \to \Sigma_g$ be a surface bundle that has a section of self-intersection $k$, where $k$ is divisible by $2h-2$. Then after stabilisation $E$ admits a flat structure whose horizontal foliation has a closed leaf of self-intersection $k$.
\end{thm} 
\begin{proof}
We first stabilise $E$ until the Milnor-Wood equality is satisfied for $S$. We let $\bar{g} = g + g'$ denote the genus of the base of the stabilisation and let $\bar{\rho}: \pi_1(\Sigma_{\bar{g}}) \to \Gamma_{h,1}$ be its holonomy representation. Since the Milnor-Wood inequality is satisfied for $S$, it has a tubular neighbourhood $\nu_S$ that is diffeomorphic to a flat $GL^+(2,\mathbb{R})$-bundle. We let $\xi$ denote the corresponding horizontal foliation on $\nu_S$. We extend $\xi$ to a horizontal distribution $\xi'$ that agrees with $\xi$ on a (possibly smaller) neighbourhood of $S$. We choose curves $a_i,b_i$ representing the standard generators of $\pi_1(\Sigma_{\bar{g}})$ and let $\phi_i, \psi_i \in  Diff^+(\Sigma_h)$ be the holonomy maps induced by $\xi'$, so that $[\phi_i] = \bar{\rho}(a_i)$ and $\ [\psi_i]  = \bar{\rho}(b_i)$ in  $\Gamma_h$. Note that these diffeomorphisms depend on the choice of curves and not just their homotopy classes.

By construction the distribution $\xi'$ is a foliation in a neighbourhood of $S$. Hence the product of commutators $\eta= \prod_{i = 1}^{\bar{g}} [\phi_i, \psi_i]$ has compact support disjoint from the marked point corresponding to the section $S$, and is thus an element in $Diff^c(\Sigma_h^1)$. We next consider the following diagram that relates the mapping class groups $\Gamma_h^1, \Gamma_{h,1}$ and $\Gamma_h$:
\[\xymatrix{  & 1 \ar[d] & 1 \ar[d]   & & \\
& \mathbb{Z} \ar[d] & \mathbb{Z} \ar[d] & &\\
1 \ar[r] & \pi_1(T_1\Sigma_h) \ar[r] \ar[d] & \Gamma_h^1 \ar[d] \ar[r] & \Gamma_h \ar[r]& 1\\
1 \ar[r] &  \pi_1(\Sigma_h) \ar[r] \ar[d] & \Gamma_{h,1} \ar[d] \ar[r] & \Gamma_h \ar[r]&1.\\
& 1 & 1 & & }\]
Here $\mathbb{Z}$ is generated by a positive Dehn twist $\Delta_{\partial}$ along an embedded curve parallel to the boundary of $\Sigma_h^1$ and $T_1\Sigma_h$ denotes the unit tangent bundle of $\Sigma_h$.

Now the image of $\eta$ in $\Gamma_{h,1}$ is trivial. Thus $\eta = \Delta_{\partial}^k$ where $k$ is the self-intersection of $S$. This is because the second column is the central extension corresponding to the vertical Euler class as a characteristic class in the group cohomology of $\Gamma_{h,1}$ (cf.\ \cite{Mor}). By assumption $k$ is divisible by $2h-2$ and hence $\eta \in H_1(\pi_1(T_1\Sigma_h)) = H_1(T_1\Sigma_h)$ is trivial. Again this is because the left most column is the central extension corresponding to the Euler class of the unit tangent bundle over $\Sigma_h$ and $[\eta] $ is a multiple of the fibre class of this $S^1$-bundle that is divisible by $2h -2$. Hence $\eta^{-1} = \prod_{j = 1}^{N} [\alpha_j, \beta_j]$ is a product of commutators in $\Gamma_h^1$, each of which lie in the kernel of the natural map to $\Gamma_h$. 

We let $\phi_j, \psi_j \in Diff^+(\Sigma_h^1)$ be representatives of the mapping classes $\alpha_j, \beta_j$ respectively, and consider the product $\gamma = \eta \thinspace \prod_{j = 1}^{N} [\phi_j, \psi_j]$ in $Diff_0^c(\Sigma_h^1)$. Since this group is perfect we may write $\gamma^{-1} = \prod_{l = 1}^{M} [\gamma_l, \delta_l]$. Letting $a_i, b_i$ be standard the  generators for $\pi_1(\Sigma_{\bar{g} +N +M})$, we define a flat bundle via the holonomy map
\[a_i \mapsto \phi_i, \ \ b_i \mapsto \psi_i \text{ for } 1 \leq i \leq \bar{g}\]
\[a_{\bar{g} +j} \mapsto \alpha_j, \ \ b_{\bar{g} +j} \mapsto \beta_j \text{ for } 1 \leq j \leq N\]
\[a_{\bar{g} +N + l} \mapsto \gamma_l, \ \ b_{\bar{g} + N + l} \mapsto \delta_l \text{ for } 1 \leq l \leq M.\]
This gives a horizontal foliation with a closed leaf of self-intersection $k$ on a stabilisation of $E$. That the bundle is a stabilisation of the original bundle follows since the mapping classes represented by $\alpha_j,\beta_j, \gamma_l,, \delta_l$ are trivial  in $\Gamma_h$.
\end{proof}
In \cite{BCS} Bestvina, Church and Souto show the non-existence of certain lifts of bundles with sections to the diffeomorphism group with marked points, using the bounds on the Euler class given by the Milnor-Wood inequality. In particular, the diagonal section in the product of two genus $h$ surfaces provides such an example. However, by Theorem \ref{leaf_stabilisation}, these examples do possess lifts after stabilisation. 

\section{Abelianisation of diffeomorphism groups with marked points}\label{points}
Our discussion above will enable us to calculate the first group homology of $Diff^+(\Sigma_{h,1})$ and, in particular, we will show that this group is not perfect. In fact it is clear that the group $Diff^+(\Sigma_{h,1})$ is not perfect, as there is a surjection to $GL^+(2, \mathbb{R})$ given by the derivative map and
\[H_1(GL_{\delta}^+(2, \mathbb{R})) = H_1((SL(2, \mathbb{R}) \times  \mathbb{R}^+)_{\delta}) =  \mathbb{R}^+, \]
since $SL(2, \mathbb{R})$ is a perfect group. But this is the only contribution to $H_1(Diff_{\delta}^+(\Sigma_{h,1}))$ if $h \geq 3$. The proof of this fact is based on the following result due to Sternberg.
\begin{thm}[Sternberg's Linearisation Theorem, \cite{Ster}]\label{Sternberg_linearisation}
Let $\phi$ be a smooth diffeomorphism defined in a neighbourhood $U$ of the origin in $\mathbb{R}^n$ and let $\phi(0) = 0$. Further, let $s_1, ...s_n \in \mathbb{C}$ denote the eigenvalues (counted with multiplicities) of the Jacobian $D_0 (\phi)$ at the origin and assume that
\[s_i \neq s_1^{m_1}...s_n^{m_n},\]
for all non-negative integers $m_1, ..., m_n$ with $\sum m_i > 1$. Then there is a change of coordinates $\psi$ that fixes the origin so that on a possibly smaller neighbourhood $W \subset U$ the following holds
\[ \psi \phi \psi ^{-1} = D_0 (\phi).\]
\end{thm}

\begin{rem}\label{Germ_version} We note that the hypotheses of Theorem \ref{Sternberg_linearisation} hold, in particular, if $D_0 (\phi) = \lambda \thinspace Id$ for $\lambda \neq 0, 1$. Sternberg's Theorem may also be interpreted in terms of germs of diffeomorphisms, i.e. if the hypotheses of the theorem are satisfied for a germ $\phi \in \mathcal{G}_p$, then $\phi$ is conjugate to the germ represented by $D_p (\phi)$.
\end{rem}
\begin{prop}\label{not_perfect}
Let $\Sigma_h$ be a surface of genus $h \geq 3$. Then $H_1(Diff_{\delta}^+(\Sigma_{h,1})) = \mathbb{R}^+$.
\end{prop}
\begin{proof}
We consider the extension given in Lemma \ref{exact_sequence}
\[ 1 \to Diff^c(\Sigma_{h,1}) \to Diff^+(\Sigma_{h,1}) \stackrel{\pi} \rightarrow \mathcal{G}_p \to 1.\]
Since the group $Diff^c(\Sigma_{h,1})$ is perfect (cf.\ Proposition \ref{horizontal_leaves}), the associated five-term exact sequence yields $H_1(Diff_{\delta}^+(\Sigma_{h,1})) = H_1(\mathcal{G}_p)$. We next consider the exact sequence
\[  1 \to \mathcal{G}_{p,Id} \to \mathcal{G}_p \stackrel{D_p} \rightarrow GL^+(2, \mathbb{R}) \to 1,\]
where $ \mathcal{G}_{p,Id} $ is the set of germs whose linear part is the identity. By Remark \ref{Germ_version} above, if $\phi \in \mathcal{G}_{p}$ and $D_p (\phi) = \lambda \thinspace Id$ for some $\lambda > 1$, then there is a $ \psi \in  \mathcal{G}_{p}$ so that $\psi \phi \psi ^{-1} = \lambda \thinspace Id$. We set $A_{\lambda} =  \lambda \thinspace Id$, then for any $\phi \in  \mathcal{G}_{p,Id}$, there is a germ $\psi$ such that
\[A_{\lambda} = \psi ( A_{\lambda} \phi ) \psi ^{-1}.\]
Since $A_{\lambda}$ is central in $GL^+(2,\mathbb{R})$, we may assume that $\psi \in \mathcal{G}_{p,Id}$ after conjugating the above equation with the element $D_p \psi$. Thus $\phi = \psi^{-1} A_{\lambda}^{-1} \psi A_{\lambda} = [\psi^{-1}, A_{\lambda}^{-1}]$ and we have shown that $H_1(\mathcal{G}_{p,Id})_{GL^+(2, \mathbb{R})} = 0$. 

In view of this, the five-term exact sequence gives
\[ 0 = H_1(G_{p,Id})_{GL^+(2, \mathbb{R})} \to H_1(\mathcal{G}_p ) \to H_1(GL_{\delta}^+(2, \mathbb{R})) \to 0\]
and, hence, $H_1(Diff_{\delta}^+(\Sigma_{h,1}))=H_1(GL_{\delta}^+(2,\mathbb{R})) = \mathbb{R}^+$.
\end{proof}
We let $PDiff^+(\Sigma_{h,k})$ denote the group of pure orientation preserving diffeomorphisms, i.e. an element $\phi \in PDiff^+(\Sigma_{h,k})$ is a diffeomorphism of $\Sigma_h$, which fixes a set of $k$ marked points \emph{pointwise}. With exactly the same argument as in Proposition \ref{not_perfect} we obtain the following.
\begin{prop}\label{not_perfect2}
Let $\Sigma_h$ be a surface of genus $h \geq 3$. Then $H_1(PDiff_{\delta}^+(\Sigma_{h,k})) = (\mathbb{R}^+)^k$.
\end{prop}
Using Proposition \ref{not_perfect2} it is now possible to compute the first homology of the full diffeomorphism group $Diff^+(\Sigma_{h,k})$. 
\begin{thm}\label{not_perfect3}
Let $\Sigma_h$ be a surface of genus $h \geq 3$ and let $k \geq 2$. Then 
\[H_1(Diff_{\delta}^+(\Sigma_{h,k})) = \mathbb{R}^+ \times \mathbb{Z}_2.\]
\end{thm}
\begin{proof}
By considering the action on the marked points induced by $Diff^+(\Sigma_{h,k})$, we obtain the following extension of groups
\[1 \to PDiff^+(\Sigma_{h,k}) \to Diff(\Sigma_{h,k}) \to S_k \to 1.\]
The five-term exact sequence for group homology then gives
\[ H_2(S_k) \stackrel{\partial} \rightarrow H_1(PDiff_{\delta}^+(\Sigma_{h,k}))_{S_k} \to H_1(Diff_{\delta}^+(\Sigma_{h,k}))  \to H_1(S_k) \to 0.\]
By Proposition \ref{not_perfect2} we have that $H_1(PDiff_{\delta}^+(\Sigma_{h,k})) = (\mathbb{R}^+)^k$, which, in particular, implies that $H_1(PDiff_{\delta}^+(\Sigma_{h,k}))_{S_k} = \mathbb{R}^+$. Since $S_k$ is a finite, the group $H_2(S_k)$ consists entirely of torsion elements. Hence, as $\mathbb{R}^+$ is torsion free the connecting homomorphism $\partial$ is trivial and we obtain the following short exact sequence:
\[0 \to \mathbb{R}^+ \to H_1(Diff_{\delta}^+(\Sigma_{h,k})) \to H_1(S_k) = \mathbb{Z}_2 \to 0.\]
Finally, since in $\mathbb{R}^+$ every element has a square root, this extension has a section and we conclude that 
\[H_1(Diff_{\delta}^+(\Sigma_{h,k})) = \mathbb{R}^+ \times  \mathbb{Z}_2. \qedhere\]
\end{proof}

The proof of Proposition \ref{not_perfect} above will allow us to calculate the first group homology of $Diff^c(\mathbb{R}^2,0)$, which here denotes the group of diffeomorphisms of the plane that have compact support and fix the origin. This fact was stated in a more general form by Fukui in \cite{Fuk}, however his argument appears to be incomplete. 

Fukui argues as follows (see \cite{Fuk}, p.\ 485). Let $\phi \in Diff^c(\mathbb{R}^n,0)$ have $D_0 \phi = Id$, then there is a product of commutators so that $\eta = \phi \prod_{i = 1}^{g'} [\alpha_i, \beta_i]$ is the identity on some neighbourhood of $0$. He then claims that by Thurston's result on the perfectness of the identity component the group of diffeomorphisms, we may write $\eta$ as a product of commutators of elements in $Diff^c(\mathbb{R}^n \setminus \{0 \})$, which denotes the group of compactly supported diffeomorphisms of $\mathbb{R}^n \setminus 0$. In order to apply the result of Thurston, one must have that $\eta$ is isotopic to the identity through diffeomorphisms with compact support away from the origin. However, it is not clear that $\eta$ is isotopic to the identity through diffeomorphisms with support disjoint from the origin. In fact for $n = 2$ the mapping class group of compactly supported diffeomorphisms on $\mathbb{R}^2 \setminus \{ 0 \}$ is isomorphic to $\mathbb{Z}$ (cf.\ \cite{Iva}, Cor.\ 2.7 E). As a corollary of the results we have obtained thus far we are now able to give a complete proof of the theorem stated by Fukui in the case $n = 2$.
\begin{thm}\label{Fukui}
$H_1(Diff^c_{\delta}(\mathbb{R}^2, 0)) = \mathbb{R}^{+}$.
\end{thm}
\begin{proof}
We have the following exact sequences
\begin{equation*}
 1 \to  Diff^c(\mathbb{R}^2 \setminus \{ 0 \}) \to  Diff^c(\mathbb{R}^2, 0) \stackrel{\pi} \rightarrow \mathcal{G}_p \to 1
\end{equation*}
and
\begin{equation*}
 1 \to  Diff_0^c(\mathbb{R}^2 \setminus \{ 0 \}) \to  Diff^c(\mathbb{R}^2 \setminus \{ 0 \}) \rightarrow \mathbb{Z} \to 1.
\end{equation*}
We consider the five-term sequence in cohomology associated to the first exact sequence above
\begin{align*}
0 \to H^1(\mathcal{G}_p)  \to H^1(Diff^c_{\delta}(\mathbb{R}^2, 0)) &\to H^1(Diff^c_{\delta}(\mathbb{R}^2 \setminus \{ 0 \}))^{\mathcal{G}_p}\\
& \stackrel{\delta} \rightarrow  H^2(\mathcal{G}_p) \to H^2(Diff^c_{\delta}(\mathbb{R}^2, 0)).
\end{align*}
By Thurston's result $Diff_0^c(\mathbb{R}^2 \setminus \{ 0 \})$ is perfect and applying the five-term exact sequence to the second exact sequence above implies that $H^1(Diff^c_{\delta}(\mathbb{R}^2 \setminus \{ 0 \})) = \mathbb{Z}$.

Next we consider the sequence of classifying spaces
\[\xymatrix{B Diff^c_{\delta}(\mathbb{R}^2, 0) \ar[r] & B \mathcal{G}_p \ar[r] &  B GL_{\delta}^+(2, \mathbb{R}) \ar[r] & B GL^+(2, \mathbb{R}).}\]
The Euler class is a generator of $H^2(B GL^+(2, \mathbb{R}))$ and the pullback to $H^2(\mathcal{G}_p)$ is non-zero and primitive, as one sees by evaluating this class on a flat $GL^+(2, \mathbb{R})$-bundle with Euler class 1, thought of as an element in $H_2(\mathcal{G}_p)$. Moreover, a flat bundle with holonomy in $ Diff^c(\mathbb{R}^2, 0)$ is topologically trivial, since it admits a section with vanishing self-intersection number. Hence the pullback of $e$ to $H^2(  Diff^c_{\delta}(\mathbb{R}^2, 0))$ is zero. By exactness of the five-term sequence above $e = \delta(f) \text{ for some } f \in H^1(Diff^c_{\delta}(\mathbb{R}^2 \setminus \{ 0 \}))^{\mathcal{G}_p} \subset \mathbb{Z}$. Hence as $e$ is a primitive, non-torsion class, the connecting homomorphism for the five-term exact sequence in \emph{homology} must be surjective. Thus by exactness
\[ H_1(Diff^c_{\delta}(\mathbb{R}^2, 0)) = H_1(\mathcal{G}_p) = \mathbb{R}^+,\]
where the second equality was shown in the proof of Proposition \ref{not_perfect}.
\end{proof}
\section{Closed leaves of flat bundles with symplectic holonomy}\label{symp_flat_leaves}
One may consider flat bundles with additional structure. In the context of surface bundles with horizontal foliations it is natural to consider bundles whose horizontal foliations are transversally symplectic. This is equivalent to the existence of a fibrewise symplectic form that is holonomy invariant, and such a bundle will be called \emph{symplectically flat}. In this way, a flat bundle $\Sigma_h \to E \to B$ with a transversal symplectic structure is equivalent to a holonomy representation $\pi_1(B) \stackrel{\rho} \rightarrow Symp(\Sigma_h, \omega)$, where $\omega$ is the symplectic form restricted to a fibre. We shall for the most part suppress any explicit reference to the symplectic form, since by Moser stability any two symplectic forms on $\Sigma_h$ are equivalent after rescaling.

In order to prove the existence of symplectically flat bundles, we wish to mimic the proof of Proposition \ref{horizontal_leaves2}. In order to do this we note that the analogue of Lemma \ref{exact_sequence} holds in the symplectic case (cf.\ \cite{Bow2}, Proposition 4.2.1). So if $\mathcal{G}^{Symp}_p$ the group of symplectomorphism germs that fix the marked point $p$ and $Symp^c(\Sigma_{h,1})$ consists of symplectomorphisms whose supports are disjoint from $p$, we have the following.
\begin{lem}\label{exact_sequence_Symp}
The following sequence of groups is exact
\[ 1 \to Symp^c(\Sigma_{h,1}) \to Symp(\Sigma_{h,1}) \stackrel{\pi} \rightarrow \mathcal{G}^{Symp}_p \to 1.\]
\end{lem}
The other important ingredient is the fact that the compactly supported symplectomorphism group is perfect. The proof of this fact is implicit in the articles \cite{KM},\cite{KM2} and for a detailed proof we refer to (\cite{Bow2}, Lemma 4.2.7).
\begin{lem}\label{Symp_perfect}
The group $Symp^c(\Sigma_h^k)$ is perfect for $h \geq 3$.
\end{lem}
\noindent With these facts the proof of the following is entirely analogous to the smooth case and the details are left to the reader.
\begin{prop}\label{symp_leaves_construction}
For $h \geq 3$, there exist flat bundles $\Sigma_h \to E \to \Sigma_g$ that have symplectic holonomy and whose horizontal foliations have arbitrarily many closed leaves $S_k$ of prescribed self-intersection $[S_k]^2 = m_k \leq h -1$.
In particular, if $m_k = 0$ we may assume that the horizontal foliation in some neighbourhood $\nu_k$ of $S_k$ is given by the kernel of a projection $\nu_k = \Sigma_h \times D^2 \to D^2$.
\end{prop}
In \cite{KM} the authors introduced the notion of a \emph{symplectic pair} (cf.\ also \cite{BK}). In the case of four-manifolds, a symplectic pair consists of a pair of complementary, two-dimensional transversally symplectic foliations. As a consequence of Proposition \ref{symp_leaves_construction} we may now give examples of manifolds with symplectic pairs, both of whose foliations have closed leaves of non-zero self-intersection.
\begin{cor}\label{Pairs_no_triv_leaves}
There exist 4-manifolds that admit symplectic pairs $(\omega_1, \omega_2)$ both of whose foliations $\mathcal{F}_1, \mathcal{F}_2$ have closed leaves $L_1, L_2$ with $[L_i]^2 \neq 0$.
\end{cor}
\begin{proof}
We let $E_1$ be a flat symplectic bundle with a section $s_1$ such that $[s_1]^2 \neq 0$. We further let $E_2$ be a flat symplectic bundle with two sections $s_2,t_2$, the first of which has non-trivial self-intersection and the second of which has a neighbourhood on which the horizontal foliation is given by projection to a disc. The existence of these bundles is guaranteed by Proposition \ref{symp_leaves_construction}. By a suitable choice of $E_1$ we may assume that the genus of its fibre $g(F_1)$ to be arbitrarily large. After stabilisation of $E_2$, we may also assume that the genus of the base of $E_2$ is $g(F_1)$. As noted in \cite{BK}, the assumptions on the foliation in a neighbourhood of $t_2$ imply that the normal connected sum
\[E_1 \#_{F_1 = t_2} E_2\]
admits a symplectic pair, whose foliations we denote by $\widetilde{\mathcal{F}}_i$. Then by construction the connect sums
\[\sigma_1 = s_1 \# F_2\]
\[\sigma_2 = s_2 \# F_1\]
are leaves of $\widetilde{\mathcal{F}}_i$ and $[\sigma_i]^2 = [s_i]^2 \neq 0$.
\end{proof}
For smooth diffeomorphisms we computed the Abelianisation of $Diff^+(\Sigma_{h,k})$. It is then natural to try to determine the Abelianisation of $Symp(\Sigma_{h,k})$, however one cannot mimic the proof of Proposition \ref{not_perfect} used above. The first step is still valid and thus one has that $H_1(Symp_{\delta}(\Sigma_{h,1})) = H_1(\mathcal{G}^{Symp}_p)$. There is also a version of the Sternberg Linearisation Theorem for symplectic germs, but the normal form that it yields is not linear (cf.\ \cite{Ster2}), and thus the computation of the Abelianisation of $Symp(\Sigma_{h,1})$ remains an open question.

\subsection*{The case of genus $0$:}\label{case_low_genus}
So far the results that we have obtained have been for bundles whose fibre $\Sigma_h$ has been of genus at least 3. We shall now consider the case of genus $0$, where one can give a fairly precise description of the possible compact leaves of a (symplectically) flat bundle.

Examples of sphere bundles with horizontal foliations that have closed leaves of arbitrary self-intersection have been given by Mitsumatsu (cf.\ \cite{Mit}). We shall summarise his construction here. Let $\mathbb{R}^2 \to \xi_k \to \Sigma_g$ be a flat bundle of Euler class $k \leq g -1$ as given by \cite{Mil}. Then the sphere bundle $ E_k = S(\xi_k \oplus \mathbb{R})$ is flat and has two sections $L_{\pm}$ corresponding to the north and south poles of the fibre and $[L_{\pm}]^2 = \pm k$.

We would of course like to have similar examples for flat bundles with symplectic holonomy. The flat structures that one obtains via the construction of Mitsumatsu cannot have symplectic holonomy. For if so, then one would have a vertical symplectic form $\omega_v$ that is positive on each fibre, i.e. $\omega_v([F]) \neq 0$, and vanishes identically on the leaves of the horizontal foliation. But the set $\{L_{-}, L_{+}\}$ generates $H_2(E_k, \mathbb{R})$, which is a contradiction. It is not difficult to adapt the argument of Proposition \ref{symp_leaves_construction} to produce horizontal foliations of sphere bundles with symplectic holonomy (cf.\ \cite{Bow2}, p.\ 76). Interestingly, this gives foliations with a unique compact leaf. For if $L'$ were any other leaf then $\{L , L' \}$ would generate $H_2(E_k, \mathbb{R})$ and this would contradict the existence of a vertical symplectic form.

\section{Filling flat $S^1$-bundles}\label{fill_poss}
Given a flat $S^1$-bundle, we would like to know when it bounds a flat surface bundle. To answer this question in full generality is a subtle matter. However for bundles over compact surfaces this can always be achieved after stabilisation.
\begin{prop}\label{boundary_Euler_class}
Let $h \geq 3$ or $h = 0$ and let $M$ be a flat $S^1$-bundle. Then there is a flat bundle $\Sigma_h^1 \to E \to \Sigma_g$, whose boundary is a stabilisation of $M$. In particular, there exist flat $\Sigma_h^1$-bundles whose boundaries have non-trivial Euler class.
\end{prop}
\begin{proof}
Let $a_i, b_i \in \pi_1(\Sigma_g)$ denote the standard generators of the fundamental group and let $ \phi_i = \rho(a_i)$ and $\psi_i = \rho(b_i) $ be the images of these generators in $Diff_0(S^1)$ under the monodromy homomorphism. Since $\phi_i, \psi_i$ are isotopic to the identity, we may extend them to diffeomorphisms $\bar{\phi}_i, \bar{\psi}_i$ on a collar of the boundary $[0, 1] \times S^1$ in such a way that
\[\bar{\phi}_i(t,x) = (t,\phi_i(x)) \text{ , } \bar{\psi}_i(t,x) = (t,\psi_i(x)) \text{ for } 0 \leq t <  \epsilon\]
and
\[\bar{\phi}_i(t,x) = \bar{\psi}_i(t,x) = Id\text{ for } 1 - \epsilon <  t \leq 1.\]
We then extend by the identity to obtain $\bar{\phi}_i, \bar{\psi}_i \in Diff^+(\Sigma_h^1)$ such that $\eta = \prod_{i = 1}^{g} [\bar{\phi}_i, \bar{\psi}_i]$ lies in $Diff^c(\Sigma_h^1)$. 

For $h \geq 3$ the group $Diff^c(\Sigma_h^1)$ is perfect (cf.\ Proposition \ref{horizontal_leaves}) and for $h = 0$ this is the classical result of Thurston (cf.\ \cite{Th2}). Thus we may write $\eta^{-1} = \prod_{i = 1}^{g'} [\alpha_i, \beta_i]$, where $\alpha_i, \beta_i \in  Diff^c(\Sigma_h^1)$. We define a flat bundle $E$ over $\Sigma_{g + g'}$ by the holonomy representation
\[a_i \mapsto \bar{\phi}_i, \ \ b_i \mapsto \bar{\psi}_i \text{ for } 1 \leq i \leq g\]
\[a_{g +j} \mapsto \alpha_j, \ \ b_{g +j} \mapsto \beta_j \text{ for } 1 \leq j \leq g'.\]
The boundary of $E$ is a flat $S^1$-bundle and by construction it is a stabilisation of $M$ as required.

For the second statement we note that there exist flat $GL^+(2,\mathbb{R})$-bundles with non-trivial Euler classes, with explicit examples described in \cite{Mil}. The associated projective space bundles are then flat $S^1$-bundles with non-trivial Euler classes and these can then be filled by flat disc bundles after stabilisation, which does not affect the Euler class.
\end{proof}
Proposition \ref{boundary_Euler_class} implies that any flat circle bundle can be filled in by a flat disc bundle after a suitable stabilisation. On the other hand, if we require that the bundle have symplectic holonomy, then this is no longer true (cf.\ Theorem \ref{Tsuboi_int}). However, if the fibre has genus $h \geq 3$, then one can indeed find a filling by a symplectically flat bundle after a suitable stabilisation. To this end we shall need an analogue of the extension trick of Proposition \ref{boundary_Euler_class} in the symplectic case.
\begin{prop}\label{flat_extension}
Let $\pi_1(\Sigma_g) \stackrel{\rho} \rightarrow Diff_0(S^1)$ be a flat structure on an $S^1$-bundle $M$ and let $\phi_i, \psi_i$ denote $\rho(a_i), \rho(b_i)$ respectively. Then there are symplectic extensions $\tilde{\phi}_i, \tilde{\psi}_i$ on the annulus $A = S^1 \times [0, 1]$ that are the identity in a neighbourhood of $S^1 \times \{ 1 \}$ such that $\prod_i^g [\tilde{\phi}_i, \tilde{\psi}_i]$ has support in the interior of $A$.
\end{prop}
\begin{proof}
Let $\mathcal{F}$ be the horizontal foliation given by the flat structure on $M$ and let $\alpha \in \Omega^1(M)$ be a defining 1-form for $\mathcal{F}$. We choose a function $\phi$ on $[0, 1]$, which is equal to $t$ on a neighbourhood of $0$ and is identically zero for all $t$ in a neighbourhood of $1$. Set $\omega = dt \wedge \alpha + \phi(t) d \alpha$ on $E = M \times [0, 1]$ and let $\frac{\partial}{\partial \theta}$ denote a vector field that is tangent to the fibres of $M$. Then
\[\omega(\frac{\partial}{\partial t},\frac{\partial}{\partial \theta}) = \alpha(\frac{\partial}{\partial \theta}) \neq 0,\]
since $\mathcal{F}$ is transverse to the fibres of $M$ and, thus, $\omega$ is a nowhere vanishing 2-form on $E$. Furthermore, since $\mathcal{F}$ is a foliation we compute:
\begin{align*}
\omega^2 &= (dt \wedge \alpha + \phi(t)d \alpha)^2\\
& = 2\phi(t) dt \wedge \alpha \wedge d \alpha = 0.
\end{align*}
Thus $\mathcal{F}_{\omega} = Ker(\omega)$ is a well-defined distribution that is transverse to the (annular) fibres of $E \to \Sigma_g$. Moreover, since $\omega = d(t \alpha)$ in a neighbourhood of $M \times \{ 0 \}$ this distribution is integrable and transversally symplectic on this neighbourhood, and restricts to $\mathcal{F}$ on $M \times \{ 0 \}$. On a neighbourhood of $M \times \{ 1 \}$ the form $\omega$ reduces to $dt \wedge \alpha$ and again the kernel distribution is integrable and agrees with $\mathcal{F}$ on this neighbourhood.

We choose a base point $x_0 \in \Sigma_g$ and embedded representatives $a_i,b_i$ for the standard generators of $\pi_1(\Sigma_g, x_0)$ and let $\bar{\phi}_i, \bar{\psi}_i$ be the holonomies of the curves $a_i,b_i$ given by the distribution $\mathcal{F}_{\omega}$. Then on $S^1 \times \{ 0 \}$ and near $S^1 \times \{ 1 \}$ these diffeomorphisms are given by $\phi_i \times Id$ and $\psi_i \times Id$ respectively, where $\phi_i, \psi_i$ are the images of the standard basis under the holonomy representation of $M$. Since $\phi_i, \psi_i$ lie in $Diff_0(S^1)$, we may alter the maps $\bar{\phi}_i, \bar{\psi}_i$ near $S^1 \times \{ 1 \}$ so that they restrict to the identity in a neighbourhood of $S^1 \times \{ 1 \}$. We shall continue to denote these altered maps by $\bar{\phi}_i, \bar{\psi}_i$.

We let $\Omega$ be the restriction of $\omega$ to the annular fibre over $x_0$. Then since the holonomies $\bar{\phi}_i, \bar{\psi}_i$ have support in $S^1 \times [0,1)$ and the distribution defining them was transversally symplectic in a neighbourhood of $M \times \{ 0 \}$, the forms $\bar{\phi}^*_i \Omega - \Omega$ and $\bar{\psi}^*_i \Omega - \Omega$ have compact support in the interior of $S^1 \times (0,1)$. Moreover
\[\int_A (\bar{\phi}^*_i \Omega - \Omega) = \int_A (\bar{\psi}^*_i \Omega - \Omega) = 0\]
so that the forms $\bar{\phi}^*_i \Omega$ and $\bar{\psi}^*_i \Omega$ are cohomologous to $\Omega$ in compactly supported cohomology. By applying a Moser isotopy, which will have support in the interior of $S^1 \times [0,1]$, we obtain symplectomorphisms $\tilde{\phi_i}, \tilde{\psi}_i$ that are symplectic extensions of $\phi_i, \psi_i$ respectively, and by construction $\prod_i^g [\tilde{\phi}_i, \tilde{\psi}_i]$ has support in the interior of $A$.
\end{proof}
Proposition \ref{flat_extension} is the main step in extending flat structures symplectically and the following result follows from this and the perfectness of $Symp^c(\Sigma_h^1)$.
\begin{thm}\label{symp_flat_extension}
Let $M$ be a flat $S^1$-bundle and assume that $h \geq 3$. Then some stabilisation of $M$ bounds a flat $\Sigma_h^1$-bundle with symplectic holonomy.
\end{thm}
\begin{proof}
Let $\pi_1(\Sigma_g) \stackrel{\rho} \rightarrow Diff_0(S^1)$ be the holonomy representation associated to $M$ and let $\tilde{\phi}_i, \tilde{\psi}_i \in Symp(A)$ be the extensions given by Proposition \ref{flat_extension}. After a suitable choice of symplectic form on $\Sigma_h^1$, we may symplectically embed $A = S^1 \times [0,1]$ in $\Sigma_h^1$ so that $S^1 \times \{ 0 \}$ maps to $\partial \Sigma_h^1$ . We then consider $\eta = \prod_i^g [\tilde{\phi}_i, \tilde{\psi}_i]$ as an element in $Symp^c(\Sigma_h^1)$. By Lemma \ref{Symp_perfect} this group is perfect and, thus, we may write $\eta^{-1}$ as a product of $g'$ commutators. We then define the associated flat bundle $E'$ over $\Sigma_{g + g'}$ as in the proof of Proposition \ref{boundary_Euler_class}, and by construction the boundary of $E'$ is a stabilisation of $M$.
\end{proof}
Theorem \ref{symp_flat_extension} can be interpreted in terms of the five-term exact sequence of a certain extension of groups. For this we let $Symp(\Sigma_h^1)$ as usual denote the group of symplectomorphisms of $\Sigma_h^1$. We further let $Symp(\Sigma_h^1, \partial \Sigma_h^1)$ denote those symplectomorphisms that restrict trivially to the boundary. Then as a consequence of Proposition \ref{flat_extension} the following sequence, which is given by restriction to $\partial \Sigma_h^1$, is exact:
\[1 \to Symp(\Sigma_h^1, \partial \Sigma_h^1) \to Symp(\Sigma_h^1) \to Diff^+(\partial \Sigma_h^1) = Diff_0(S^1) \to 1.\]
With this notation we have the following proposition.
\begin{prop}\label{conn_surj}
For $h \geq 3$ the connecting homomorphism in the five-term exact sequence in real cohomology associated to the following exact sequence is trivial:
\[1 \to Symp(\Sigma_h^1, \partial \Sigma_h^1) \to Symp(\Sigma_h^1) \to Diff^+(\partial \Sigma_h^1) = Diff_0(S^1) \to 1.\]
\end{prop}
\begin{proof}
By the Universal Coefficient Theorem it suffices to show that the map
\[H_2(Symp_{\delta}(\Sigma_h^1)) \to H_2( Diff_{0,{\delta}}(S^1))\]
is surjective on integral cohomology. This follows immediately from Theorem \ref{symp_flat_extension}, since any flat $S^1$-bundle extends after stabilisation and this does not change the homology class represented by this bundle in $H_2( Diff_{0,{\delta}}(S^1))$.
\end{proof}
The Godbillon-Vey class of the horizontal foliation of a flat $S^1$-bundle $M$ defines an element $GV$ in $H^2( Diff_{0,{\delta}}(S^1), \mathbb{R})$, which is non-trivial by the work of Thurston (cf.\ \cite{Bott}). It is possible that the Godbillon-Vey class provides an obstruction to the existence of a flat symplectic bundle $E$ that bounds $M$. However, by Proposition \ref{conn_surj} the image of the class $GV$ in $H^2(Symp_{\delta}(\Sigma_h^1), \mathbb{R})$ is non-trivial. Geometrically, this means that after stabilisation the horizontal foliation of any $S^1$-bundle extends to a transversally symplectic foliation on some surface bundle $E$ with fibre $\Sigma_h^1$. In particular, the Godbillon-Vey class is not an obstruction to finding a null-cobordism that extends the horizontal foliation of $M$ to the interior of $E$ symplectically.

\section{Flat bundles and the extended Hamiltonian group}\label{Flat_Ham}
In Section \ref{fill_poss} we showed that any flat circle bundle over a surface can be filled by a flat disc bundle with smooth holonomy after stabilisation. However, as was shown in \cite{Tsu}, it is not in general possible to fill in a flat circle bundle by a flat disc bundle that has symplectic holonomy, even after stabilisation, since the existence of such a filling implies that the Euler class of the circle bundle vanishes. More specifically Tsuboi proved the following theorem.
\begin{thm}[\cite{Tsu}]\label{Tsuboi}
Let $\pi_1(\Sigma_g) \stackrel{\psi} \rightarrow Diff_0(S^1)$ be a homomorphism and let $a_i, b_i$ be standard generators of $\pi_1(\Sigma_g)$. Furthermore, let $f_i, h_i \in Symp(D^2)$ be extensions of $\psi(a_i), \psi(b_i)$ respectively and let $e(E)$ denote the Euler class of the total space of the $S^1$-bundle $E$ associated to $\psi$. Then 
\[-\pi^2 e(E) = Cal([f_1, h_1]...[f_g, h_g]).\]
\end{thm}
In the case of a disc, a diffeomorphism is symplectic if and only if it is Hamiltonian. For bundles with fibres of higher genus we shall generalise Tsuboi's result under the assumption that the holonomies are Hamiltonian. As usual a symplectomorphism $\psi \in Symp_0(\Sigma_h^1)$ will be called Hamiltonian if it is isotopic to the identity via an isotopy $\psi_t$ such that $\iota_{\dot{\psi}_t} \omega = dH_t$ for $0 \leq t \leq 1$. As a first step, following \cite{Tsu} we note that it is always possible to extend diffeomorphisms on the boundary of $\Sigma_h^1$ to its interior by a Hamiltonian diffeomorphism. It suffices to consider the case $M = [0, 1) \times S^1$ and to show that any diffeomorphism of the boundary extends to a Hamiltonian diffeomorphism on $M$ that has compact support. The following is essentially Lemma 2.2 of \cite{Tsu} and the proof will be omitted. 
\begin{lem}\label{extend_volume}
Let $M = [0, 1) \times S^1$ and let $\omega$ be a symplectic form of finite total volume. Then the map $Ham^c(M) \to Diff_0(S^1)$ given by restriction is surjective.
\end{lem}
As a consequence of Lemma \ref{extend_volume} we obtain an exact sequence:
\begin{equation*}
1 \to Ham(\Sigma_h^1, \partial \Sigma_h^1) \to Ham(\Sigma_h^1) \to Diff^+(\partial \Sigma_h^1) = Diff_0(S^1) \to 1,
\end{equation*}
where $Ham(\Sigma_h^1, \partial \Sigma_h^1)$ denotes the intersection $Symp(\Sigma_h^1, \partial \Sigma_h^1) \cap Ham(\Sigma_h^1)$.

Exactly as in the case of compactly supported symplectomorphisms, we define a Flux homomorphism $Symp_0(\Sigma_h^1) \to H^1(\Sigma_h^1, \mathbb{R})$ via the formula
\[ Flux(\psi) = \int_0^1(\iota_{\dot{\psi}_t} \omega) \thinspace dt, \]
where $\psi_t$ is an isotopy from $Id$ to $\psi$. As in the compactly supported case, one can show that $Flux(\psi) = [\lambda - \psi^* \lambda]$ for any primitive $ \lambda$ such that $\omega = -d \lambda$ (cf.\ \cite{McS}, Lemma 10.14). Hence, $Flux$ is well-defined independently of the choice of isotopy $\psi_t$ and primitive $\lambda$. Moreover, it is easy to show that $Ker(Flux) = Ham(\Sigma_h^1)$ (cf.\ \cite{Bow2}, Lemma 5.2.4).

One may also define a Calabi homomorphism $Cal$ on $Ham(\Sigma_h^1, \partial \Sigma_h^1)$. For this let $\lambda$ be a primitive such that $\omega = - d \lambda$ and define
\[Cal(\phi) = -\frac{1}{3} \int_{\Sigma_h^1} \phi^* \lambda \wedge \lambda .\]
Again this definition is independent of the choice of $\lambda$ (see \cite{McS}, Lemma 10.26).

We will now extend Tsuboi's result to bundles with fibre $\Sigma_h^1$. In order to do this we shall need to reinterpret Theorem \ref{Tsuboi} in terms of a five-term exact sequence. Now the map $Cal$ is an element of $H^1(Ham_{\delta}(D^2, \partial D^2),\mathbb{R})$ and we claim that it is invariant under the conjugation action of $Ham(D^2)$. For let $\psi \in Ham(D^2)$ and $\phi \in  Ham(D^2, \partial D^2)$, and let $\lambda$ be a primitive such that $ \omega = - d \lambda$. Then $\psi^* \lambda$ is also a primitive for $\omega$ and we have
\begin{equation*}
\int_{D^2} \phi^* \lambda \wedge \lambda =  \int_{D^2} \phi^* (\psi^* \lambda) \wedge (\psi^*\lambda) =  \int_{D^2} (\psi \phi \psi^{-1})^* \lambda \wedge \lambda. \end{equation*}
Thus, $Cal \in H^1(Ham_{\delta}(D^2, \partial D^2),\mathbb{R})^{Diff_0(S^1)}$ and we claim that Tsuboi's result can be interpreted as saying that the image of $Cal$ under the connecting homomorphism in the five-term exact sequence is a non-zero multiple of the Euler class $e$ considered as an element in the real group cohomology of $Diff_0(S^1)$. For this we need to make use of an explicit description of the connecting homomorphism, which is described in Appendix \ref{App_five_term} below.
\begin{thm}\label{Tsuboi_five}
Consider the extension of groups
\[1 \to Ham(D^2, \partial D^2) \to Ham(D^2) \to Diff_0(S^1) \to 1,\]
and let $\delta$ denote the connecting homomorphism in the five-term exact sequence in real cohomology. Then $\delta [Cal]$ is $-\pi^2 \thinspace e$, where $e$ denotes the Euler class in $H^2(Diff_{0,\delta}(S^1), \mathbb{R})$.
\end{thm}
\begin{proof}
In order to verify the equality $\delta Cal = -\pi^2 \thinspace e$ in real cohomology, it suffices to evaluate both sides on 2-cycles $Z$ in $H_2(Diff_{0,\delta}(S^1))$. Such a cycle may be thought of as the image of the fundamental class under the map induced by a representation of a surface group $\pi_1(\Sigma_g) \stackrel{\psi} \rightarrow Diff_0(S^1)$. If we let $a_i, b_i$ be standard generators of $\pi_1(\Sigma_g)$, then a generator of $H_2(\pi_1(\Sigma_g))$  may be described by the group 2-cycle
\begin{align*}
 z & = (a_1,b_1) + (a_1b_1,a_1^{-1}) + ... + (a_1b_1...b_{g-1}a_{g}^{-1}, b_g^{-1})\\
&  - (2g+1)(e,e) - \sum_{i = 1}^g (a_i,a_i^{-1}) + (b_i,b_i^{-1}) .
\end{align*}
Since $[a_1, b_1]...[a_g, b_g] = e$ in $\pi_1(\Sigma_g)$, we compute that
\begin{align*} \partial z  =  \sum_{i = 1}^g (a_i) + (a_i^{-1}) + (b_i) +(b_i^{-1}) &- \sum_{i = 1}^g [(a_i) - (e) +(a_i^{-1}) + (b_i) - (e) +(b_i^{-1})] \\ 
& - 2g(e) + ([a_1, b_1]...[a_g, b_g]) -(e)  = 0.
\end{align*}
We let $f_i, h_i$ denote representatives of $\psi(a_i), \psi(b_i)$ in $Diff_0(S^1)$ considered as a quotient group, and let $\tilde{z}$ be the associated lift of the fundamental cycle above. Then we compute
\[\partial \tilde{z} = ([f_1, h_1]...[f_g, h_g]) - (e).\]
Thus, by Lemma \ref{connecting_hom}, there is a set-theoretic extension $Cal_S$ of $Cal$ to $Ham(D^2)$ such that
\begin{align*}
  \delta Cal (Z) & =  \delta Cal_S(\psi_* [\Sigma_g]) = Cal_S(\partial \tilde{z}) \\
& = Cal_S(([f_1, h_1]...[f_g, h_g]) - (e))= Cal([f_1, h_1]...[f_g, h_g]),
\end{align*}
and by Proposition \ref{extended_Tsuboi} this is equal to $-\pi^2 \thinspace e(E)$.
\end{proof}
\noindent We are now ready to generalise Theorem \ref{Tsuboi} to the case of surfaces of higher genus.
\begin{prop}\label{extended_Tsuboi}
Let $\pi_1(\Sigma_g) \stackrel{\psi} \rightarrow Diff_0(S^1)$ be a homomorphism and let $a_i,b_i$ be standard generators of $\pi_1(\Sigma_g)$. Let $f_i, h_i \in Ham(\Sigma_h^1)$ be any extensions of $\psi(a_i), \psi(b_i)$ respectively and let $e(E)$ denote the Euler class of the total space of the $S^1$-bundle $E$ associated to $\psi$. Then 
\[-\pi^2 e(E) = Cal([f_1, h_1]...[f_g, h_g]).\]
\end{prop}
\begin{proof}
By Lemma \ref{extend_volume} we may assume that the extensions $f_i, h_i$ are Hamiltonian and have support in a collar $K = [0,1) \times S^1$ of the boundary. We may then consider $K \subset D^2$ with an appropriately chosen area form $\Omega$ on $D^2$ and $f_i, h_i$ as elements in $Diff_{\Omega}(D^2)$. We then compute
\begin{equation*}
 \psi^* \delta Cal([\Sigma_g]) = Cal^{\Sigma^1_h}([f_1, h_1]...[f_{g}, h_{g} ]) = Cal^{D^2}([f_1, h_1]...[f_{g}, h_{g} ]),
\end{equation*}
where the first equality follows as in Theorem \ref{Tsuboi_five} and the second follows from our choice of extensions. The latter value is $-\pi^2 \thinspace e(E)$ by Theorem \ref{Tsuboi_five}. Thus, since the left hand side of the equation  above is independent of any choices, we conclude that for \emph{any} extensions $f_i,h_i$
\[Cal^{\Sigma^1_h}([f_1, h_1]...[f_{g}, h_{g} ]) = -\pi^2 \thinspace e(E). \qedhere\]
\end{proof}
In particular, it follows by the exactness of the five-term sequence that the boundary of any flat bundle with holonomy in $Ham(\Sigma_h^1)$ is trivial as an $S^1$-bundle. Furthermore, with our interpretation of Tsuboi's result we may extend our discussion to extended Hamiltonian groups as introduced in \cite{KM}.  To this end we define $\widetilde{Flux}_{\lambda}: Symp(\Sigma_h^1) \to H^1(\Sigma_h^1, \mathbb{R})$ by 
\[\widetilde{Flux}_{\lambda}(\phi) = [(\phi^{-1})^*\lambda - \lambda] \]
for some fixed primitive $ - d \lambda = \omega$. This map is a crossed homomorphism and is referred to as an \emph{extended Flux homomorphism}, since is restricts to the ordinary Flux map on $Symp_0(\Sigma_h^1)$. We note that the definition of $\widetilde{Flux}_{\lambda}$ depends in an essential way on the choice of primitive $\lambda$. For if $\lambda '$ is another primitive, then $\lambda - \lambda' = \alpha$ is closed and
\begin{equation}\label{Extended_Flux_1}
\widetilde{Flux}_{\lambda}(\phi) = \widetilde{Flux}_{\lambda'}(\phi) + [(\phi^{-1})^* \alpha - \alpha].
\end{equation}
In terms of group cohomology this means that $\widetilde{Flux}_{\lambda}$ and $\widetilde{Flux}_{\lambda'}$ are cohomologous, when considered as elements in $H^1(Symp(\Sigma_h^1),H^1(\Sigma_h^1, \mathbb{R}))$. The \emph{extended Hamiltonian group} $\widetilde{Ham}(\Sigma_h^1)$  is defined as the kernel of $\widetilde{Flux}_{\lambda}$, which is a subgroup since $\widetilde{Flux}_{\lambda}$ is a crossed homomorphism. The group $Ham(\Sigma_h^1)$ is contained in $ \widetilde{Ham}(\Sigma_h^1)$ and we may extend the Calabi homomorphism to a map $\widetilde{Cal}_{\lambda}$ on the group 
$\widetilde{Ham}(\Sigma_h^1, \partial \Sigma_h^1) = Symp(\Sigma_h^1, \partial \Sigma_h^1) \cap \widetilde{Ham}(\Sigma_h^1)$ by defining 
\[\widetilde{Cal}_{\lambda}(\phi) =  -\frac{1}{3} \int_{\Sigma_h^1} \phi^* \lambda \wedge \lambda =  \frac{1}{3} \int_{\Sigma_h^1} (\phi^{-1})^* \lambda \wedge \lambda,\]
where $\lambda$ is the primitive chosen in the definition of $\widetilde{Flux}_{\lambda}$. This is a homomorphism on $\widetilde{Ham}(\Sigma_h^1, \partial \Sigma_h^1)$, since the following holds on $Symp(\Sigma_h^1, \partial \Sigma_h^1)$ (cf.\ \cite{KM2}, Prop.\ 19):
\begin{equation}\label{Calabi_prop_1_extended}
\widetilde{Cal}_{\lambda}(\phi \psi) = \widetilde{Cal}_{\lambda}(\phi) + \widetilde{Cal}_{\lambda}(\psi) +\frac{1}{3} \widetilde{Flux}_{\lambda}(\phi) \wedge (\phi^{-1})^*\widetilde{Flux}_{\lambda}(\psi).
\end{equation}
Again the definition of $\widetilde{Cal}_{\lambda}$ depends on the choice of primitive $\lambda$. However, we do have the following technical lemma, which will be important in showing the equivariance of $\widetilde{Cal}_{\lambda}$.
\begin{lem}\label{Cal_invariance}
Let $\phi \in \widetilde{Ham}_{\lambda}(\Sigma_h^1, \partial \Sigma_h^1) \cap \widetilde{Ham}_{\lambda'}(\Sigma_h^1, \partial \Sigma_h^1)$ for two different primitives $\lambda, \lambda'$ and set $\alpha = \lambda - \lambda'$, further assume that $\phi^* \alpha - \alpha = d H_{\phi}$ is exact. Then $\widetilde{Cal}_{\lambda}(\phi) = \widetilde{Cal}_{\lambda'}(\phi)$.
\end{lem}
\begin{proof}
By assumption $ \lambda - (\phi^{-1})^*\lambda$ is exact and hence $\phi^* \lambda - \lambda = dF_{\phi}$ is also exact. Since the boundary of $\Sigma_h^1$ is connected we may further assume that $F_{\phi} = 0$ on $\partial \Sigma_h^1$. We compute
\begin{align*}
\widetilde{Cal}_{\lambda}(\phi) & = -\frac{1}{3} \int_{\Sigma_h^1} \phi^* \lambda \wedge \lambda = -\frac{1}{3} \int_{\Sigma_h^1} (\phi^* \lambda -  \lambda) \wedge \lambda\\
& = -\frac{1}{3} \int_{\Sigma_h^1} dF_{\phi} \wedge \lambda  = -\frac{1}{3} \int_{\Sigma_h^1} d(F_{\phi}\lambda) - F_{\phi}\wedge d \lambda = -\frac{1}{3} \int_{\Sigma_h^1} F_{\phi} \omega \text.\\
\end{align*}
Similarly one computes
\[\widetilde{Cal}_{\lambda'}(\phi) = \widetilde{Cal}_{\lambda}(\phi)   -\frac{1}{3} \int_{\Sigma_h^1} \phi^* \alpha \wedge \alpha = -\frac{1}{3} \int_{\Sigma_h^1} F_{\phi}\omega - \frac{1}{3} \int_{\Sigma_h^1} H_{\phi} \omega,\]
where $\phi^* \alpha - \alpha = dH_{\phi}$ and again we assume that $H_{\phi} = 0$ on $\partial \Sigma_h^1$. We finally have that
\begin{align*}
- \int_{\Sigma_h^1} H_{\phi} \omega & = \int_{\Sigma_h^1} H_{\phi} d \lambda = \int_{\Sigma_h^1} d(H_{\phi}\lambda) - d H_{\phi}\wedge \lambda\\
& = \int_{\Sigma_h^1} ( \alpha - \phi^* \alpha) \wedge \lambda = \int_{\Sigma_h^1} \alpha \wedge \lambda  - \phi^* \alpha \wedge \lambda\\
& = \int_{\Sigma_h^1} \phi^* \alpha \wedge \phi^* \lambda - \phi^* \alpha \wedge \lambda = \int_{\Sigma_h^1} \phi^* \alpha \wedge (\phi^* \lambda - \lambda)\\
& = 0
\end{align*}
since $\lambda - \phi^* \lambda = dF_{\phi}$, $d \alpha = 0$ and $F_{\phi}|_{\partial \Sigma_h^1} = 0$.
\end{proof}
\noindent Lemma \ref{Cal_invariance} then implies that $\widetilde{Cal}_{\lambda}$ is equivariant under the conjugation action of $\widetilde{Ham}(\Sigma_h^1)$.
\begin{cor}\label{Cal_equi}
Let $\psi \in \widetilde{Ham}_{\lambda}(\Sigma_h^1)$, then $\widetilde{Cal}_{\lambda} = \widetilde{Cal}_{\psi^* \lambda}$. In particular, $\widetilde{Cal}_{\lambda}$ is equivariant under the action of $\widetilde{Ham}_{\lambda}(\Sigma_h^1)$.
\end{cor}
\begin{proof}
Let $\lambda' = \psi^* \lambda$, then $\lambda'$ is also a primitive with $- d\lambda' = \omega$ and $\alpha = \lambda' - \lambda = dH_{\psi}$ is exact since $\psi \in \widetilde{Ham}_{\lambda}(\Sigma_h^1)$. By formula (\ref{Extended_Flux_1}), it follows that $\widetilde{Flux}_{\lambda} = \widetilde{Flux}_{\lambda'}$ and hence
\[\widetilde{Ham}_{\lambda}(\Sigma_h^1, \partial \Sigma_h^1) = \widetilde{Ham}_{\lambda'}(\Sigma_h^1, \partial \Sigma_h^1).\]
By applying Lemma \ref{Cal_invariance} we conclude that $\widetilde{Cal}_{\lambda} = \widetilde{Cal}_{\psi^* \lambda}$. This then implies
\begin{align*}
\widetilde{Cal}_{\lambda}(\phi)  & = -\frac{1}{3} \int_{\Sigma_h^1} \phi^* \lambda \wedge \lambda = -\frac{1}{3} \int_{\Sigma_h^1} \phi^* (\psi^*\lambda) \wedge \psi^* (\lambda)\\
& = -\frac{1}{3} \int_{\Sigma_h^1} (\psi \phi \psi^{-1})^*\lambda \wedge \lambda = \widetilde{Cal}_{\lambda}(\psi \phi \psi^{-1}),
\end{align*}
which is the desired equivariance.
\end{proof}
We may now prove an analogue of Tsuboi's result for the extended Hamiltonian group. For the sake of notational expediency, we shall drop any explicit references to $\lambda$.
\begin{thm}\label{Hamiltonian_extended_Tsuboi}
Let $\pi_1(\Sigma_g) \stackrel{\psi} \rightarrow Diff_0(S^1)$ be a homomorphism and let $a_i,b_i$ be standard generators of $\pi_1(\Sigma_g)$. Let $f_i, h_i \in \widetilde{Ham}(\Sigma_h^1)$ be any extensions of $\psi(a_i), \psi(b_i)$ respectively and let $e(E)$ denote the Euler class of the total space of the $S^1$-bundle $E$ associated to $\psi$. Then 
\[-\pi^2 e(E) = Cal([f_1, h_1]...[f_g, h_g]).\]
In particular, if $\Sigma_h^1 \to E \to \Sigma_g$ is a flat bundle with holonomy in the extended Hamiltonian group, then the boundary is a trivial bundle.
\end{thm}
\begin{proof}
We consider the commuting diagram
\[\xymatrix{1 \ar[r] & Ham(\Sigma_h^1, \partial \Sigma_h^1) \ar[r] \ar[d]& Ham(\Sigma_h^1)  \ar[r] \ar[d] & Diff_0(S^1)   \ar[r] \ar[d] & 1\\
1 \ar[r]& \widetilde{Ham}(\Sigma_h^1, \partial \Sigma_h^1)  \ar[r] & \widetilde{Ham}(\Sigma_h^1)  \ar[r] &  Diff_0(S^1)  \ar[r]  & 1.}\]
The five-term sequence then gives the following commuting triangle
\[\xymatrix{ H^1(Ham_{\delta}(\Sigma_h^1, \partial \Sigma_h^1), \mathbb{R})^{Diff_0(S^1)} \ar[r]^<<<<{\delta} & H^2(Diff_{0,\delta}(S^1), \mathbb{R})\\
H^1(\widetilde{Ham}_{\delta}(\Sigma_h^1, \partial \Sigma_h^1), \mathbb{R})^{Diff_0(S^1)} \ar[u] \ar[ur]^{\delta} &}.\]
By Corollary \ref{Cal_equi} we see that $\widetilde{Cal} \in H^1(\widetilde{Ham}_{\delta}(\Sigma_h^1, \partial \Sigma_h^1), \mathbb{R})^{Diff_0(S^1)}$. If $\iota$ denotes the inclusion of $ Ham(\Sigma_h^1, \partial \Sigma_h^1)$ into $\widetilde{Ham}(\Sigma_h^1, \partial \Sigma_h^1)$, then $\iota^* \widetilde{Cal} = Cal$ and, hence, $\delta(\widetilde{Cal}) = \delta(Cal)$ is also a multiple of the Euler class by Proposition \ref{extended_Tsuboi}.

The second statement follows from the exactness of the five-term sequence.
\end{proof}
A comparison of Theorem \ref{symp_flat_extension} and Theorem \ref{Hamiltonian_extended_Tsuboi} exhibits a stark difference between the two groups $Symp(\Sigma_h^1)$ and $\widetilde{Ham}(\Sigma_h^1)$.

\section{The Calabi map and the first MMM-class}\label{MMM_first}
We have seen that the Euler class of the boundary of a surface bundle with one boundary component can be interpreted as the image of the Calabi map under the connecting homomorphism of a certain five-term exact sequence. We shall give a similar construction for the first Mumford-Miller-Morita (MMM) class $e_1$, which represents a generator of $H^2(\Gamma^1_h, \mathbb{R}) \cong \mathbb{R}$ for $h \geq 3$. We recall that the first MMM-class of a surface bundle over a surface is up to a constant just the signature of the total space.

In order to describe the first MMM-class in terms of a five-term exact sequence we shall need to consider compactly supported extended Hamiltonian groups. We first define an extended Flux homomorphism $\widetilde{Flux}_c$ on $Symp^c(\Sigma_h^1) \subset Symp(\Sigma_h^1)$ by restricting the map $\widetilde{Flux}_{\lambda}$ defined in Section \ref{Flat_Ham}. There are other possible extensions of $Flux$ to crossed homomorphisms, but on the the level of group cohomology these can be easily described and the following is a slight variant of (\cite{KM2}, Theorem 11).
\begin{lem}\label{Flux_class}
Let $\widetilde{Flux}$ be an extended Flux homomorphism on $Symp^c(\Sigma_h^1)$, then the following holds in  $H^1(Symp^c(\Sigma_h^1), H_c^1(\Sigma_h^1, \mathbb{R}))$
\[[\widetilde{Flux}] = [\widetilde{Flux}_c] +  a[p^*k_{\mathbb{R}}],\]
where $a \in \mathbb{R}$ and $k_{\mathbb{R}} \in H^1(\Gamma_h^1, H_c^1(\Sigma_h^1, \mathbb{R})) \cong \mathbb{R}$ is the generator defined by the extended Johnson homomorphism of Morita.
\end{lem}
\begin{proof}
We let $\Delta = \widetilde{Flux} - \widetilde{Flux}_c$. For $\phi \in Symp_0^c(\Sigma_h^1)$ and $\psi \in Symp^c(\Sigma_h^1)$ we see that on the level of cochains
\begin{align*}
\Delta(\phi.\psi) & = [\widetilde{Flux} - \widetilde{Flux}_c](\phi.\psi)\\
& = [\widetilde{Flux}(\phi) -  \widetilde{Flux}_c(\phi)] + (\phi^{-1})^*[\widetilde{Flux}(\psi) -  \widetilde{Flux}_c(\psi)]\\
& = [Flux(\phi) -  Flux(\phi)] + [\widetilde{Flux}(\psi) -  \widetilde{Flux}_c(\psi)]\\
& = [\widetilde{Flux}(\psi) -  \widetilde{Flux}_c(\psi)]  = \Delta(\psi).
\end{align*}
Moreover, $\Delta$ vanishes on $Symp_0^c(\Sigma_h^1)$ by definition and is coclosed, hence $[\Delta] = p^* [\beta]$ for some $[\beta] \in H^1(\Gamma^1_h, H_1(\Sigma_h^1, \mathbb{R}))$. Now this group is isomorphic to $\mathbb{R}$ and is generated by the extended Johnson homomorphism of Morita (cf.\ \cite{KM2} and \cite{Mor2}).
\end{proof}
\noindent In view of Lemma \ref{Flux_class} we set 
\[\widetilde{Flux_a} = \widetilde{Flux}_c + a \thinspace p^* k_{\mathbb{R}}, \text{ for any } a \in \mathbb{R}.\]
The kernel of $\widetilde{Flux_a}$ will then be denoted by $\widetilde{Ham_a^c}(\Sigma_h^1)$ and by considering the projection to the mapping class group we obtain the following extension of groups
\begin{equation*}
1 \to Ham^c(\Sigma_h^1) \to \widetilde{Ham^c}(\Sigma_h^1) \stackrel{p} \rightarrow \Gamma_h^1 \to 1.
\end{equation*}
For any group $G$ there is a pairing $H^1(G, H_c^1(\Sigma_h^1, \mathbb{R})) \times H^1(G, H_c^1(\Sigma_h^1, \mathbb{R})) \to H^2(G, \mathbb{R})$, which we denote $[\alpha.\beta]$ for classes $[\alpha], [\beta] \in H^1(G, H_c^1(\Sigma_h^1, \mathbb{R}))$. This is defined via the following formula
\[\alpha.\beta(\phi,\psi) = \alpha(\phi) \wedge (\phi^{-1})^*\beta(\psi).\]
The induced map on cohomology is well-defined independently of the chosen representatives, and is natural with respect to pullbacks (cf.\ \cite{KM2}, Lemma 18). With these preliminaries we may now prove the following theorem.
\begin{thm}\label{MMM_Cal}
Let $h \geq 2$ and let 
\[1 \to Ham^c(\Sigma_h^1) \to \widetilde{Ham_a^c}(\Sigma_h^1) \stackrel{p} \rightarrow \Gamma_h^1 \to 1\]
be the extension associated to the extended Hamiltonian group defined by the extended flux map $\widetilde{Flux_a}$.
Then the image of $[Cal]$ under the connecting homomorphism of the associated five-term exact sequence is $ \frac{1}{3}a^2 e_1$. In particular, if $a$ is non-zero, then any flat bundle with holonomy in $\widetilde{Ham_a^c}(\Sigma_h^1)$  has signature zero.
\end{thm}
\begin{proof}
We consider the following part of the five-term exact sequence associated to the extended Hamiltonian group
\[H^1(\widetilde{Ham_{a,\delta}^c}(\Sigma_h^1),\mathbb{R}) \to H^1(Ham_{\delta}^c(\Sigma_h^1),\mathbb{R})^{\Gamma_h^1} \stackrel{\delta} \rightarrow H^2(\Gamma_h^1, \mathbb{R}).\]
We first claim that the Calabi map lies in the invariant part of $H^1(Ham^c(\Sigma_h^1),\mathbb{R})$. For if $\lambda$ is a primitive for the symplectic form $\omega$ on $\Sigma_h^1$, then so is $\psi^* \lambda$ for any $\psi \in \widetilde{Ham_a^c}(\Sigma_h^1)$. Thus, since the Calabi map is independent of the choice of primitive for any $\phi \in Ham^c(\Sigma_h^1)$, we compute that
\begin{align*}
Cal(\phi) & = -\frac{1}{3} \int_{\Sigma_h^1} \phi^* \lambda \wedge \lambda = -\frac{1}{3} \int_{\Sigma_h^1} \phi^*( \psi^* \lambda) \wedge (\psi^*\lambda)\\
& =  -\frac{1}{3} \int_{\Sigma_h^1}(\psi \phi \psi^{-1})^* \lambda \wedge \lambda = Cal(\psi\phi\psi^{-1}),
\end{align*}
and $[Cal]$  lies in $H^1(Ham^c(\Sigma_h^1),\mathbb{R})^{\Gamma_h^1}$ as claimed. 

If $i$ denotes the inclusion $\widetilde{Ham_a^c}(\Sigma_h^1) \hookrightarrow Symp^c(\Sigma_h^1)$, then by definition $\widetilde{Flux_a}$ vanishes on $\widetilde{Ham^c_a}(\Sigma_h^1)$ and we see that $i^*[\widetilde{Flux}_c]  = - a  . i^*(p^* [k_{\mathbb{R}}])$. Let $\widetilde{f} = \widetilde{Cal}_{\lambda}$ and note that by formula (\ref{Calabi_prop_1_extended}) this map satisfies the hypotheses of Lemma \ref{connecting_hom2}. Thus we have an explicit description of the connecting homomorphism in terms of $\widetilde{Cal}_{\lambda}$. More precisely,  let $\overline{\phi},\overline{\psi} \in \Gamma_h^1$ be considered as elements of the quotient, then we compute
\begin{align*}
\delta(Cal)(\overline{\phi},\overline{\psi}) & = \widetilde{Cal}_{\lambda}(\phi) + \widetilde{Cal}_{\lambda}(\psi) -  \widetilde{Cal}_{\lambda}(\phi.\psi)\\
& = - \frac{1}{3} \widetilde{Flux}_c(\phi) \wedge (\phi^{-1})^*\widetilde{Flux}_c(\psi)\\
& = - \frac{1}{3} a . i^*(p^* k_{\mathbb{R}})(\phi) \wedge (\phi^{-1})^*a . i^*(p^* k_{\mathbb{R}})(\psi)\\
& = - \frac{1}{3} a^2  k_{\mathbb{R}}(\overline{\phi}) \wedge (\phi^{-1})^*k_{\mathbb{R}}(\overline{\psi}).
\end{align*}
Hence we have shown that $[\delta(Cal)] = -\frac{1}{3} a^2 [k_{\mathbb{R}}.k_{\mathbb{R}}]$. Now we know by (\cite{Mor1}, Proposition 4.2) that $[k_{\mathbb{R}}.k_{\mathbb{R}}] = - e_1$, where $e_1$ is the first MMM-class, and we conclude that $\delta[Cal] =  \frac{1}{3}a^2 e_1$. The second claim follows by the exactness of the five-term exact sequence.
\end{proof}
We may now give an interpretation of the signature of certain surface bundles in terms of the Calabi map of commutators lying in the kernel of a particular extended Flux homomorphism. Specifically, we let $\widetilde{Flux}$ be the pullback of the extended flux map on $Symp(\Sigma_h)$ under the inclusion $Symp^c(\Sigma^1_h) \hookrightarrow Symp(\Sigma_h)$. By Theorem 12 of \cite{KM2} we know that $[\widetilde{Flux}] = [\widetilde{Flux}_c] + p^* [k_{\mathbb{R}}]$. Then as a consequence of Theorem \ref{MMM_Cal} and the calculations in the proof of Theorem \ref{Tsuboi_five} we obtain the following corollary.
\begin{cor}\label{Tsuboi_first_MMM}
Let $\Sigma^1_h \to E \to \Sigma_g$ be a bundle with holonomy representation $\pi_1(\Sigma_g) \stackrel{\rho} \rightarrow \Gamma_h^1$ and assume $h \geq 2$. Furthermore, let $\alpha_i = \rho(a_i)$ and $\beta_i = \rho(b_i)$ be the images of standard generators $a_i,b_i$ of $\pi_1(\Sigma_g)$. Then for any lifts $\phi_i, \psi_i \in \widetilde{Ham_1^c}(\Sigma_h^1)$ of $\alpha_i, \beta_i$ the signature satisfies
\[\sigma(E) = Cal([\phi_1, \psi_1]...[\phi_g, \psi_g ]).\]
\end{cor}
We contrast Corollary \ref{Tsuboi_first_MMM} with results of \cite{KM}, where it is shown that there exists $Symp^c(\Sigma_g^1)$-bundles with non-zero signature. In particular, we see that those bundles can not have holonomy in the subgroup $\widetilde{Ham_1^c}(\Sigma_g^1)$.

\appendix
\section{Five-term exact sequences}\label{App_five_term}
To any extension of groups $1 \to  N  \to  G \to Q \to  1$ one may associate a five-term exact sequence is group cohomology of the following form:
\[\xymatrix{ 1 \ar[r]  & H^1(Q,R)  \ar[r]  & H^1(G,R) \ar[r]  & H^1(N,R)^Q \ar[r]^{\delta} & H^2(Q,R)\ar[r] & H^2(G,R)}.\]
This exact sequence is generally derived by means of the Hochschild-Serre spectral sequence, but we choose to give an alternate description in order to obtain an explicit formula for the connecting homomorphism. It can be shown that the connecting homomorphism that one obtains in this way agrees with the usual one up to sign, at least for cohomology with real coefficients, but for a detailed account of this and the results below we refer to (\cite{Bow2}, Appendix A).
\begin{lem}\label{connecting_hom}
Let $1 \to N \to G \stackrel{\pi} \rightarrow Q \to 1$ be an extension of groups and let $S$ denote a normalised set-theoretic section of the final map so that $s(e.N) = e$. Further let $\phi \in H^1(N, R)^Q$ lie in the invariant part of $H^1(N, R)$ for any coefficient ring $R$. Define
\[\phi_S(g) = \phi(n_g) + f(s(N.g))\]
where $n_g \in N$ is the unique element such that $g = n_g.s(N.g)$ and $f$ is any function on the set of coset representatives determined by $s$. 

Then the map $\delta: H^1(N, R) \to H^2(Q,R)$ defined by $\delta \phi = \pi(\delta \phi_S)$ is well-defined and the five-term sequence is exact with $\delta$ as the connecting homomorphism. Furthermore if $\frac{1}{2} \in R$, then we may assume that $\phi_S(g^{-1}) = - \phi_S(g)$.
\end{lem}
The five-term exact sequence is natural in the following sense.
\begin{lem}\label{bar_delta_nat}
Consider the following commuting diagram of group extensions
\[\xymatrix{ 1 \ar[r] & N  \ar[r] & G \ar[r]& Q \ar[r]  & 1\\
1 \ar[r] & N'  \ar[r] \ar[u] & G' \ar[r] \ar[u]& Q \ar[r] \ar[u]& 1.}\]
Then there is a commutative diagram of five-term exact sequences:
\[\xymatrix{   H^1(Q,R)  \ar[r] \ar[d]& H^1(G,R) \ar[r] \ar[d]   & H^1(N,R)^Q \ar[r]^{\delta}  \ar[d] & H^2(Q,R)\ar[r] \ar[d] & H^2(G,R) \ar[d]\\
H^1(Q,R)  \ar[r] & H^1(G',R) \ar[r]  & H^1(N',R)^{Q} \ar[r]^{\delta} & H^2(Q,R) \ar[r] & H^2(G',R).}\]
\end{lem}
There is a slightly different formulation of Lemma \ref{connecting_hom} that is useful in performing calculations.
\begin{lem}\label{connecting_hom2}
Let $1 \to N \to G\rightarrow Q \to 1$ be an extension of groups and let $f \in H^1(N,R)^Q$. Further let $\widetilde{f}$ be an extension of $f$ to $G$ such that $\widetilde{f}(n.g) = \widetilde{f}(n) + \widetilde{f}(g)$ for all $n \in N$ and $g \in G$. Then for any $[g_1],[g_2] \in Q$ in the quotient, there is a representative cocycle for $\delta \phi \in H^2(Q,R)$ such that
\[\delta f ([g_1],[g_2]) = \widetilde{f}(g_1) + \widetilde{f}(g_2) - \widetilde{f}(g_1.g_2).\]
\end{lem}

\end{document}